\documentclass{amsart}
\usepackage{graphicx}
\usepackage{multicol,multirow}
\usepackage{amsmath,amssymb,amsfonts}
\usepackage{mathrsfs}
\usepackage{rotating}
\usepackage{appendix}
\usepackage[numbers]{natbib}

\theoremstyle{oupplain}
\newtheorem{theorem}{Theorem}[section]
\newtheorem{lemma}[theorem]{Lemma}
\newtheorem{proposition}[theorem]{Proposition}

\theoremstyle{oupdefinition}
\newtheorem{definition}{Definition}[section]
\theoremstyle{oupremark}
\newtheorem{remark}[theorem]{Remark}

\theoremstyle{oupproof}

\numberwithin{equation}{section}
\usepackage{amsmath,amsthm,mathtools,amssymb,textcomp,extarrows,bm,mleftright}
\mleftright
\usepackage[T1]{fontenc}
\usepackage{calligra}
\RequirePackage{mathrsfs}
\let\mathcal\mathscr

\usepackage{xparse}
\ExplSyntaxOn
\NewDocumentCommand{\definealphabet}{mmmm}
{
	\int_step_inline:nnn { `#3 } { `#4 }
	{
		\cs_new_protected:cpx { #1 \char_generate:nn { ##1 }{ 11 } }
		{
			\exp_not:N #2 { \char_generate:nn { ##1 } { 11 } }
		}
	}
}
\ExplSyntaxOff
\definealphabet{bb}{\mathbb}{A}{Z}
\definealphabet{cal}{\mathcal}{A}{Z}
\definealphabet{frak}{\mathfrak}{A}{z}
\definealphabet{rm}{\mathrm}{A}{z}
\definealphabet{bf}{\mathbf}{A}{Z}

\usepackage{dsfont}

\begin{document}
	
\title{Compatibility of theta lifts and tempered condition}

	\author{Zhe Li}
	\address{School of mathematical Sciences, Fudan University, 220 Handan Rd., Yangpu District,200433,Shanghai, China}
	\curraddr{Department of Mathematics and Statistics, Case Western Reserve University, Cleveland, Ohio 43403}
	\email{zli17@fudan.edu.cn}
	
	\author{Shanwen WANG}
	\address{School of Mathematics, Renmin University of China, No. 59 Zhongguancun Street, Haidian District, 100872, Beijing, China}
	\email{s\_wang@ruc.edu.cn}
	\thanks{Shanwen Wang is supported by the Fundamental Research Funds for the Central Universities, the Research Funds of Renmin University of China No.2020030251 and The National Natural Science Foundation of China (Grant No.2020010209).}

	
	
	
	\begin{abstract}In this note, we show that the metaplectic theta correspondence is compatible with the tempered condition by directly estimating the matrix coefficients, without using the classification theorem.
		
	\end{abstract}
	
	\maketitle

\section{Introduction}
Throughout this article, we fix an additive character $\psi$ of $\mathbb{R}$. Let $dx$ be the unique Haar measure on $\mathbb{R}$ which is selfdual for Fourier transformation with respect to $\psi$.  Unless we explicitly mention the contrary, by a representation, we always mean a unitary Casselman-Wallach representation of finite length (Fr\'echet representation of moderate growth), cf. \cite[Chapter XII]{Wal92}. The inner product on a representation is denoted by $( -,-)$. Let $\pi$ be a representation. We denote by $\pi^{\vee}$ the space of continuous linear functionals on $\pi$ and it is given the strong topology (uniform convergence on bounded subsets). The smooth dual of $\pi$, i.e. the subspace of smooth vectors in $\pi^{\vee}$, is identified with $\overline{\pi}$.

Let ${\bf H}$ be a real reductive group $G$ or its double cover $\widehat{G}$. Among the irreducible  (genuine) representations of ${\bf H}$, there is an important class of representations, whose matrix coefficients are controlled by the Harish-Chandra $\Xi$ function. Such a representation is called tempered (genuine) representation. The classification of tempered representations is given in \cite{KZ82}, and in particular, one knows that an irreducible representation is tempered if and only if it is an irreducible parabolic induction of limit of discrete series (cf. \cite[theorem 14.2]{KZ82}). We will denote by ${\rm Temp}^{\rm irr}_{\psi}({\bf H})$ the set of irreducible tempered (genuine) representations of ${\bf H}$. 

Let $(G, G')$ be a reductive dual pair in $\mathrm{Sp}_{2m}(\mathbb{R})$. Let $\widehat{G}$ and $\widehat{G'}$ be the inverse images of $G$ and $G'$ in the metaplectic double covering group $\widehat{\rm Sp}_{2m}(\mathbb{R})$ by the covering map. For irreducible admissible representations $\pi$ and $\pi^{'}$ of $\widehat{G}$ and $\widehat{G'}$ respectively, we say $\pi$ and $\pi'$ correspond if $\pi\otimes \pi'$ is a quotient of the Weil representation $\omega$ of $\widehat{\mathrm{Sp}}_{2m}(\mathbb{R})$, restricted to $\widehat{G}\times\widehat{G'}$. Note that the Weil representation is not a representation by our convention as it is not of finite length.   

In this article, we will focus on the type ${\rm I}$ dual pairs over the field $(K,{\sharp})$ of equal rank, where $(K,{\sharp})$ is the field $\mathbb{R}$ with trivial involution $\sharp$ or the quaternion algebra ${\bbH}$ with quaternionic conjugation $\sharp$.
Let $(W, V)$ be the underlying quadratic space over $(K, \sharp)$ of our dual pair of equal rank $n$. Then,
\begin{enumerate}
	\item[(A)] If $K=\mathbb{R}$, then $W$ is a $2n$-dimensional real symplectic vector space and $V$ is a $2n+1$ real orthogonal spaces. Let ${\rm Rep}^{\rm gen}_{\psi}(\widehat{\rm Sp}(W))$ be the set of irreducible genuine representations $\pi_W$ of $\widehat{{\rm Sp}}(W)$. Let $S_{2n+1}$ be the set of isomorphism classes of real orthogonal spaces $V'$ with $\dim V'=2n+1$ and ${\rm disc}(V')\equiv 1 \in \mathbb{R}^{\times}/\mathbb{R}^{\times 2}$. Let ${\rm Rep}_{\psi}^{\rm irr}({\rm SO}(V'))$ be the set of irreducible representations of ${\rm SO}(V')$ with $V'\in S_{2n+1}$.  
	Adam and Barbasch \cite{AB98} show that the dual pair $({\rm Sp}(W)), {\rm O}(p,q))$ with $p+q=2n+1$ gives rise to a bijection between the genuine representations of metaplectic group and the representations of odd special orthogonal group of the same rank. 
	\begin{theorem}There is a bijection given by the metaplectic theta correspondence:
		\[{\rm Rep}_{\psi}^{\rm gen}(\widehat{\rm Sp}(W))\leftrightarrow \coprod_{V'\in S_{2n+1}}{\rm Rep}_{\psi}^{\rm irr}({\rm SO}(V')).\]
	\end{theorem}
	\item[(B)] If $K$ is the quaternion algebra $\bbH$, then $W$ is a skew-Hermitian space over $\bbH$ of rank $n$ or $n-1$ and $V$ is a Hermitian space over $\bbH$ of rank $n$. Let $S^{'}_{n}$ be the set of isomorphism classes of Hermitian spaces $W$ of rank $n$. Let ${\rm Rep}^{\mathrm{irr}}_{\psi}(\rm Sp(W))$ be the set of irreducible representations $\pi_W$ of ${\rm Sp}(W)$, for $W\in S^{'}_n$. Let ${\rm O}(V)$ be the isometry group of hermitian space $V=\mathbb{H}^n$ over $\bbH$ of rank $n$ and  let ${\rm Rep}^{\rm irr}_{\psi}({\rm O}(V))$ be the set of irreducible representations of ${\rm O}(V)$.
	In \cite{LPTZ03}, the authors show the following theorem.
	\begin{theorem}[\cite{LPTZ03} Theorem 5.8]
		There is a bijection given by the theta correspondence: 
		\[\coprod_{W\in S^{'}_{n}\cup S^{'}_{n-1}}\mathrm{Rep}_{\psi}^{\mathrm{irr}}(\mathrm{Sp}(W))\leftrightarrow \mathrm{Rep}_{\psi}^{\mathrm{irr}}({\rm O}(V)).\] 
	\end{theorem}
	
\end{enumerate}

Moreover, in \cite{AB98} and  \cite{LPTZ03}, the authors explicitly determined the $K$-types of all the representations on both side of the theta correspondence. Together with the classification of the irreducible tempered representations, one can deduce that the theta correspondence is compatible with the tempered condition. 
\begin{theorem}\label{main} With the same notation as above. \begin{enumerate}
		\item If $K=\mathbb{R}$, then there is a bijection given by the theta correspondence:
		\[{\rm Temp}_{\psi}^{\rm gen}(\widehat{\rm Sp}(W))\leftrightarrow \coprod_{V' \in S_{2n+1}}{\rm Temp}_{\psi}^{\rm irr}({\rm SO}(V')).\]
		\item If $K=\bbH$, then there is a bijection given by the theta correspondence:
		\[\coprod_{W\in S^{'}_{n}\cup S^{'}_{n-1}}{\rm Temp}_{\psi}^{\rm irr}(\rm Sp(W))\leftrightarrow {\rm Temp}_{\psi}^{\rm irr}({\rm O}(V)).\]
	\end{enumerate}
\end{theorem}
The main purpose of this article is to prove the theorem \ref{main} by directly estimating the matrix coefficients, without using the classification theorem. The approach of estimating matrix coefficients is widely used. In fact, Gan and Ichino use the estimations of matrix coefficients to prove the convergence of the inner product of two matrix coefficients of representation obtained by the theta correspondence for the dual pair $({\rm O}_4, {\rm Sp}_4)$ over $p$-adic field (cf.\cite[Section 9, Lemma 9.1]{GI}). In \cite[Lemma D]{GI11}, they extend this to other dual pairs over $p$-adic field : $({\rm U}(n), {\rm U}(m))$, $(\mathrm{Sp}_{2n},{\rm O}_{2m+1})$, $(\mathrm{Sp}_{2n},{\rm O}_{2m})$, and prove that in their setting, the discrete series condition is compatible with the theta correspondence. In the unitary case over $\mathbb{R}$, a corresponding estimation is given by Xue (cf.\cite[Lemma 3.2]{Xue22}). In this note, we adapt their estimations to our case to prove the temperedness of the representations is compatible with theta correspondence. Note that the dual pairs in this note are not considered by them. 
\begin{remark}
	\begin{enumerate}
		\item  In \cite{He05}, H. He proved the theta correspondence is compatible with unitary condition in the semistable range. Our setting, the equal rank case, is contained in the semistable range. But since the category of tempered representations is a subcategory of the category of unitary representations, we need to refine the estimations of matrix coefficients given by He \cite[Theorem 6.2.1, 6.3.1, 6.4.1-3)]{He05} to achieve our goal.  More precisely, to show the theta correspondence is compatible with the tempered condtion, we will need the $L^{2+\epsilon}$-convergence of the matrix coefficients to prove our results.
		\item In \cite[Lemma D]{GI11}, Gan and Ichino only provide an estimation of the matrix coefficient from big group to small group over $p$-adic field. In this note, we only provide an estimation from small group to big group over field of real numbers. If we don't assume the symmetry of the theta correspondence, our estimation can be viewed as a complement for the estimation of Gan and Ichino. 
	\end{enumerate}
\end{remark}

\section{Tempered (genuine) representations}
Let ${\bf H}$ be a real reductive group $G$ or the double covering group $\widehat{G}$ of $G$. Let $A_G$ be the maximal $\mathbb{R}$-split torus of $G$ of rank $r$ (i.e. $A_G(\mathbb{R})\cong (\mathbb{R}^{\times})^r$), and $M$ be the centralizer of $A_G$ in $G$, which is exactly the Levi factor of a minimal parabolic subgroup $P$ of $G$. We will write an element $a\in A_G(\mathbb{R})$ by $(a_1,\cdots,a_r)$. 
We denote by $\Delta=R(A_G, P)$ the set of roots of $A_G$ in the unipotent radical $U$ of $P$. Set  
\begin{equation}
	\begin{split}
		A_G^+&=\{a\in A_G(\mathbb{R}): \vert\alpha(a)\vert\leq 1, \forall \alpha\in \Delta\}\\
		&=\{(a_1,\cdots,a_r):0<\vert a_1\vert\leq \vert a_2\vert\leq\cdots\leq\vert a_r\vert\leq 1 \} .
	\end{split}
\end{equation}

We denote by $\delta_{P, G}$ the modulus character of $P$. We fix a special maximal compact subgroup $K$ of $G(\mathbb{R})$ and we have a Cartan decomposition of $G(\mathbb{R})$: \[G(\mathbb{R})=KA_G^+K.\] 
For any integrable function $f$ on $G(\mathbb{R})$, the following formula holds (cf. \cite[\S 4]{II10}):
\begin{equation}\label{deco}\int_{G(\mathbb{R})} f(g)dg=\int_{A_G^+}\nu(a)\int_{K\times K}f(k_1ak_2)dk_1dk_2da,\end{equation} where $\nu$ is a positive function on $A_G^+$ such that $\nu(a)\leq C\cdot\delta^{-1}_{P, G}(a)$ for some constant $C$.

Harish-Chandra defined a special spherical function $\Xi^G$ on $G(\mathbb{R})$, which can be used to control the growth of $C^{\infty}$-functions on $G(\mathbb{R})$ with values in $\mathbb{C}$. We recall briefly its definition and some useful results.

We denote by $C^{\infty}(G(\mathbb{R}))$ the space of all complex-valued $C^{\infty}$-functions on $G(\mathbb{R})$.	Consider the normalized smooth induced representation
\[i_{P}^{G}(1)^{\infty}:=\{f\in C^{\infty}(G(\mathbb{R})): f(pg)=\delta_{P,G}(p)^{1/2}f(g), \forall p\in P(\mathbb{R}), g\in G(\mathbb{R})\}\] 
equipped with the scalar product
\[(f, f')=\int_Kf(k)\overline{f'(k)}dk, \forall f, f'\in i_{P}^{G}(1)^{\infty}.\]
Let $e_K\in i_{P}^{G}(1)^{\infty}$ be the unique function such that $e_K(k)=1$ for all $k\in K$. Then the Harish-Chandra spherical function $\Xi^{G}$ is defined by 
\[\Xi^{G}(g)=(i_{P} ^{G}(1)(g)e_K, e_K), \forall g\in G(\mathbb{R}).\]
Note that if $f$ and $g$ are positive functions on a set $X$, we will say $f$ is essential bounded by $g$, if there exists a $c>0$ such that $f(x)\leq c g(x)$ for all $x\in X$. We will denote it by $f\ll g$. We say $f$ and $g$ are equivalent if $f$ is essentially bounded by $g$ and $g$ is essentially bounded by $f$. The function $\Xi^{G}$ is a bi-$K$-invariant function and it is independent of the choice of the maximal compact subgroup $K$ up to equivalence.

Fix an embedding $\iota: G(\mathbb{R})\rightarrow {\rm GL}_{m}(\mathbb{R})$, we define the height function
\[\sigma(g)=1+\sup\{\log\vert a_{i,j}\vert, \log \vert b_{i,j}\vert\},\]
where $(a_{i,j})$ is the matrix $\iota(g)$ and $(b_{i,j})$ is the corresponding matrix of $\iota(g^{-1})$. In particular, if $a=(a_1,\cdots, a_r)\in A_G^+$, we have 
\begin{equation}\label{height}\sigma(a)=1-\log\vert a_1\vert\geq 1.\end{equation}

We have the following well-known estimation of $\Xi^G$ due to Harish-Chandra.
\begin{lemma}\cite[theorem 30]{Var77}\label{est}
	There exists constants $A, B>0$ such that for any $a\in A_{G}^+$, we have 
	\[A^{-1}\delta_{P,G}^{\frac{1}{2}}(a)\leq \Xi^{G}(a)\leq A\delta_{P,G}^{\frac{1}{2}}(a)\sigma (a)^B.\]
\end{lemma}
The double covering group $\widehat{G}$ of $G$ is not an algebraic group, but behaves in many way like an algebraic group. In particular, we have the Cartan decomposition for $\widehat{G}$, i.e. $\widehat{G}=KA_{\widehat{G}}^+K$, where $K$ is the inverse image of a special maximal compact subgroup of $G$ and $A_{\widehat{G}}^+$ is the inverse image of $A^{+}_G$ in $\widehat{G}$. We define the corresponding Harish-Chandra spherical function by 
$\Xi^{\widehat{G}}=\Xi^{G}\circ p$, where $p$ is the covering map. 

Using Harish-Chandra's $\Xi$-function, we have the following definition of tempered representation for real reductive groups and metaplectic groups. 
\begin{definition} We say that a unitary representation $(\pi, \mathscr{H}_{\pi})$ of ${\bf H}$ is tempered if for any $e, e'\in \pi$, we have an inequality 
	\[\vert(\pi(g)e, e')\vert\ll\Xi^{\bf H}(g), \forall g\in {\bf H}(\mathbb{R}) .\] 
\end{definition}
Thanks to the work of Cowling, Haagerup and Howe \cite{CHH}, a representation of ${\bf H}$ is tempered if and only if its matrix coefficients are almost square integrable functions (i.e. it belongs to $L^{2+\epsilon}({\bf H}(\mathbb{R}))$ for all $\epsilon\in\mathbb{R}_{>0}$).

Let $\pi$ be a tempered representation of ${\bf H}$. For any $v,v'\in\pi$ and $g\in {\bf H}(\mathbb{R})$, by definition of tempered representation, there exists a constant $A_1>0$, such that 
\[\vert(\pi(g)v,v')\vert\leq A_1\cdot \Xi^{\bf H}(g).\] 
Moreover, a more precise estimation is given by
Sun \cite{Sun09} :
there is a continuous seminorm $\nu_\pi$ on $\pi$ such that
\begin{equation}\label{matrixcoeff}\vert (\pi(g)v, v')\vert\leq \Xi^{\bf H}(g)\nu_\pi(v)\nu_{\pi}(v'), \forall v, v'\in \pi. \end{equation}
We deduce, from the lemma \ref{est} and the fact that the Harish-Chandra function $\Xi^{\bf H}$ is bi-$K$-invariant, that for any $g=k_1ak_2\in KA_{\bf H}^+K$, we have
$$\Xi^{\bf H}(g)=\Xi^{\bf H}(k_1ak_2)=\Xi^{\bf H}(a).$$
If $\mathbf{H}=G$, there exists two positive constants $A_2$ and $B_1$ such that $$\Xi^{\bf H}(a)\leq A_2\delta_{P, G}^{1/2}(a)\sigma(a)^{B_1}.$$ If $\mathbf{H}=\widehat{G}$, there exists two positive constants $A_3$ and $B_2$ such that \begin{equation}\label{EST2}
	\Xi^{\bf H}(a)=\Xi^{G}(p(a))\leq A_3\delta_{P, G}^{1/2}(p(a))\sigma(p(a))^{B_2}.
\end{equation}
Thus, for any $g=k_1ak_2\in KA_{\bf H}^+K$, there exists two positive constants $A$ and $B$ such that   
\begin{equation}\label{estmeta}\vert(\pi(g)v,v')\vert\leq A\delta_{P, G}^{1/2}(a)\sigma(a)^B.\end{equation}  Here we abuse $a$ for $p(a)$ when $a\in\widehat{G}$. \section{Theta Correspondence}
In \cite[theorem 6.1]{Li89}, Li shows that if the dual pair $(G_1,G_2)$ is in the stable range, then there is an explicit realization of the theta correspondence using the mixed model of Weil representation \cite[\S 4]{Li89}.  The explicit realization of theta correspondence for unitary case is studied in \cite{LZ98} and for more general classical groups, it is studied in \cite{GQT14} and used by Xue in \cite{Xue22}. The explicit theta correspondence for our dual pairs has been described in \cite{AB98} and \cite{LPTZ03}.  In this paragraph, we recall the explicit theta correspondence using the mixed model of the Weil representation and  study the matrix coefficients of the explicit theta lift. 
\subsection{Mixed model of Weil representations}
The mixed model of Weil representations for our dual pairs are defined as follows:
\begin{enumerate}
	\item[(A)]
	Let $(W,\langle,\rangle_W)$ be a $2n$-dimensional real symplectic vector space and let $(V,\langle,\rangle_V)$ be a real quadratic space of dimension $2n+1$ with discriminant 
	\[{\rm disc}(V)=(-1)^n{\rm det}(V)\equiv 1 \in \mathbb{R}^{\times}/\mathbb{R}^{\times 2}.\]
	The space $(W\otimes V, \langle-,-\rangle_W\otimes\langle-,-\rangle_V)$ is a real symplectic space. We have a natural homomorphism 
	\begin{equation}\label{thetacorr}
		\widehat{\rm Sp}(W)\times {\rm O}(V)\rightarrow \widehat{\rm Sp}(W\otimes V).
	\end{equation}
	We denote by $H(W\otimes V)=(W\otimes V)\ltimes\mathbb{R} $ the Heisenberg group associated to the symplectic space $W\otimes V$. Let $\omega_{\psi}$ be the Weil representation of $\widehat{{\rm Sp}}(W\otimes V)\ltimes H(W\otimes V)$ associated to $W\otimes V$. We denote by $\omega_{W, V,\psi}$ the representation of $\widehat{\rm Sp}(W)\times {\rm O}(V)$ by pulling back the Weil representation $\omega_{\psi}$ by the homomorphism (\ref{thetacorr}). 
	
	Let $r_{V}$ be the Witt index of $V$. Let $V_0$ be the anisotropic kernel of $V$, which is of dimension $2n+1-2r_{V}$. Let $P_{V}=M_{V} N_{V}$ be a minimal parabolic subgroup of ${\rm O}(V)$ stabilizing a full flag of $V^{\perp}$. Let $A_{V}\cong (\mathbb{R}^{\times})^{r_{V}}$ be the maximal split torus in $M_{V}$ and define $$A^+_{V}=\{(b_1,\cdots,b_{r_{V}})\vert 0<b_1\leq\cdots\leq b_{r_{V}}\leq 1   \}.$$ 
	
	We have two dual pairs $(\mathrm{Sp}(W),{\rm O}(V))$ and $(\mathrm{Sp}(W),{\rm O}(V_0))$.
	Let $\mathscr{S}_{0}$ be the Schr\"odinger model of the Weil representation $\omega_{W,V_0,\psi}$ of the dual pair $(\mathrm{Sp}(W),{\rm O}(V_0))$. Let $\mathscr{S}=\mathscr{S}(W^{r_{V}})\widehat{\otimes}\mathscr{S}_{0}$. Then the Weil representation $\omega_{W,V,\psi}$ for the dual pair $(\mathrm{Sp}(W),{\rm O}(V))$ can be realized on $\mathscr{S}$, called the mixed model of $\omega_{W,V,\psi}$. We view elements in $\mathscr{S}$ as Schwartz functions on $W^{r_{V}}$ valued in $\mathscr{S}_{0}$.

	Since $\mathrm{Sp}(W)$ is split, the maximal split torus $A_W\cong \mathbb{R}^n$ and we define \[A_W^+=\{(a_1,\cdots,a_{n})\vert 0<a_1\leq\cdots\leq a_{n}\leq 1  \}.\]
	For any $a\in A_W^+$, $b\in A_V^+$ and $\phi\in\mathscr{S}$, we have 
	\begin{equation}\label{realaction}\omega_{W,V,\psi}(a,b)\phi(z,w)=\det(a)^{\frac{2n+1}{2}}\phi(b^{-1}za,wb). \end{equation}
	
	\item[(B)]Let $(V,(, )^{\sharp})$ be a $n$-dimension Hermitian space over $\bbH$ and let $(W, \langle,\rangle^{\sharp})$ be a $m$-dimensional skew-hermitian space $W$ over $({\bbH}, \sharp)$ with $m=n$ or $n-1$. The space 
	$(W\otimes_{\mathbb{H}}V,\mathrm{Tr}_{\mathbb{H}/\mathbb{R}}(\langle,\rangle^{\sharp}\otimes (,)^\sharp))$ is a real symplectic space of dimension $4mn$. This defines an embedding of dual pair $(\mathrm{Sp}(W),{\rm O}(V))$:
	$$ \mathrm{Sp}(W)\times {\rm O}(V)\rightarrow \mathrm{Sp}_{4nm}(\mathbb{R}). $$ Let $\omega_{W,V,\psi}$ be the oscillator representation for the dual pair $(\mathrm{Sp}(W),{\rm O}(V))$, which is a representation of $\widehat{\mathrm{Sp}}_{4nm}(\mathbb{R})$. Let $r_W$ and $r_V$ be the Witt index of $W$ and $V$ respectively. Let $W_0$ and $V_0$ be the corresponding anisotropic kernel of $W$ and $V$. Then we have 
	$$\dim_{\bbH}(W_0)=m-r_{W}, \text{ and }\dim_{\bbH}(V_0)=n-r_{V}. $$ Let $P_{W}=M_{W} N_{W}$ be a minimal parabolic subgroup of $\mathrm{Sp}(W)$ stabilizing a full flag of $W_0^{\perp}$. Then $M_{W}\cong \mathrm{GL}_1(\mathbb{R})^{r_{W}}\times \mathrm{Sp}(W_0)$. Let $A_{W}\cong (\mathbb{R}^{\times})^{r_{W}}$ be the maximal split torus in $M_{W}$ and let $$A^+_{W}=\{(a_1,\cdots,a_{r_{W}})\vert 0<a_1\leq\cdots\leq a_{r_{W}}\leq 1   \}.$$ 
	We have three dual pairs 
	\[(\mathrm{Sp}(W),{\rm O}(V)), (\mathrm{Sp}(W_0),{\rm O}(V)) \text{ and } (\mathrm{Sp}(W_0),{\rm O}(V_0)).\]
	Let $\mathscr{S}_{00}$ be the Schr\"odinger model of the Weil representation $\omega_{W_0,V_0,\psi}$ of the dual pair $(\mathrm{Sp}(W_0),{\rm O}(V_0))$. Let $\mathscr{S}_0=\mathscr{S}(W_0^{r_{V}})\widehat{\otimes}\mathscr{S}_{00}$. Then the Weil representation $\omega_{W_0,V,\psi}$ for the dual pair $(\mathrm{Sp}(W_0),{\rm O}(V))$ can be realized on $\mathscr{S}_0$. Finally the Weil representation $\omega_{W,V,\psi}$ of the dual pair $(\mathrm{Sp}(W),O(V))$ can be realized on $\mathscr{S}=\mathscr{S}(V^{r_{W}})\widehat{\otimes}\mathscr{S}_{0}$, called the mixed model of the Weil representation $\omega_{W,V,\psi}$. We view elements in $\mathscr{S}$ as Schwartz functions on $V^{r_{W}}\times W_0^{r_{V}}$ valued in $\mathscr{S}_{00}$.
	Define $$A^+_{V}=\{(b_1,\cdots,b_{r_{V}})\vert 0<b_1\leq\cdots\leq b_{r_{V}}\leq 1   \}.$$ 
	For any $a\in A_W^+$, $b\in A_V^+$ and $\phi\in\mathscr{S}$, we have \begin{equation}\label{quaternionaction}
		\omega_{W,V,\psi}(a,b)\phi(z,w)=\det(a)^{n}\det(b)^{m-r_W}\phi(b^{-1}za,wb). 
	\end{equation}
	
\end{enumerate}

\subsection{Matrix coefficients of Weil representations}

In this paragraph, we give the estimation of matrix coefficients of Weil representations $\omega_{W, V, \psi}$ using the mixed model described in the previous paragraph. 
\begin{lemma}
	For $\phi,\phi'\in\mathscr{S}(\mathbb{R})$ and $t\in\mathbb{R}^{\times}$, there exists some constant $C$ such that \begin{equation}\label{Essch}
		\left\vert\int_{\mathbb{R}}\phi(tx)\phi'(x)dx\right\vert \leq C\cdot \Upsilon(t),
	\end{equation}where $C$ is a constant and $\Upsilon(t)=\begin{cases}
		1,& \text{ if }\vert t\vert\leq 1,\\ \vert t\vert^{-1},&\text{ if }\vert t\vert >1. \end{cases}$
\end{lemma}
\begin{proof} By changing of variable, one can reduce to show that for $\phi,\phi'\in\mathscr{S}(\mathbb{R})$, there exists a constant $C$ such that 
	\[\left\vert\int_{\mathbb{R}}\phi(tx)\phi'(x)dx\right\vert \leq C,\]
	for all $0<\vert t\vert \leq 1$. It follows from a direct estimation of the integration for three regions: $\vert x\vert\leq 1$, $1<\vert x\vert \leq 1/t$ and $\vert x\vert \geq 1/t$.
\end{proof}
Using this lemma, we get the following important estimation.
\begin{proposition}For our type $I$ dual pair $(G_1, G_2)$ over $(K, \sharp)$ of equal rank $n$ with underlying spaces $(W, V)$, there exist a constant $C$, such that 
	\begin{enumerate}
		\item[(A)] If $K=\bbR$, then for $(\hat{g},h)\in \widehat{\mathrm{Sp}}(W)\times \mathrm{O}(V)$ and $\phi,\phi'\in \omega_{W,V,\psi}$, we have \begin{equation}\label{estWeil}
			\vert (\omega_{W,V,\psi}(\hat{g},h)\phi,\phi')\vert\leq C\cdot\prod_{i=1}^{n} \vert a_i\vert^{\frac{2n+1}{2}}\prod_{k=1}^{n}\prod_{j=1}^{r_V}\Upsilon(a_kb_j^{-1}).
		\end{equation} 
		\item[(B)]If $K=\bbH$, then for $(g,h)\in \mathrm{Sp}(W)\times O(V)$ and $\phi,\phi'\in \omega_{W,V,n}$, we have \begin{equation}\label{estWeil2}
			\vert (\omega_{W,V,\psi}(g,h)\phi,\phi')\vert\leq C\cdot\prod_{i=1}^{r_W} \vert a_i\vert^{n}\prod_{j=1}^{r_V}\vert b_j\vert^{m-r_W}\prod_{i=1}^{r_W}\prod_{j=1}^{r_V}\Upsilon(a_ib_j^{-1}).
		\end{equation} 
	\end{enumerate}
\end{proposition}
\begin{proof}The two cases can be proved by the same argument and we only show the case $(A)$. For any $(\widehat{g},h)\in \widehat{{\rm Sp}}(W)\times {\rm O}(V)$, by Cartan decompostion, we can write 
	$(\widehat{g}, h)=(k_1ak_2, k_1^{'}bk_2^{'})$, with $k_i$ in the inverse image $K_W$ of a special maximal compact subgroup of ${\rm Sp}(W)$, $k_i^{'}$ in a special maximal compact subgroup $K_V$ of ${\rm O}(V)$, $a\in A_W^+$ and $b\in A_V^+$. Thus for $\phi,\phi'\in\omega_{W,V,\psi}$ there exists some constant $C_1$ such that
	$$ \vert(\omega_{W,V,\psi}(\widehat{g},h)\phi,\phi')\vert\leq C_1\cdot\det(a)^{\frac{2n+1}{2}}(\phi(b^{-1}\cdot a),\phi'). $$
	Together with the previous lemma, we get the disired estimation. 
\end{proof}

In the unitary case, the above estimation is refined by by Xue \cite{Xue22}. Similarly, we can provide a more precise estimation for our dual pairs.   

\begin{proposition} 
	For our type $I$ dual pair $(G_1, G_2)$ over $(K, \sharp)$ of equal rank $n$ with underlying spaces $(W, V)$,  there exists a continuous semi-norm $\nu_{\mathscr{S}}$ on $\omega_{W, V, \psi}$ such that  
	\begin{enumerate}
		\item[(A)] If $K=\bbR$, then we have
		\begin{equation}\label{mcXue}
			\vert(\omega_{W, V,\psi}(\widehat{g}, h)\phi, \phi') \vert\leq \prod_{i=1}^{n}\vert a_i\vert^{\frac{2n+1}{2}}\prod_{k=1}^{n}\prod_{j=1}^{r_V}\Upsilon(a_kb_j^{-1})\nu_{\mathscr{S}}(\phi)\nu_{\mathscr{S}}(\phi').
		\end{equation}
		\item[(B)]If $K=\bbH$, then we have
		\begin{equation}\label{mcXue2}
			\vert (\omega_{W,V,\psi}(g,h)\phi,\phi')\vert\leq \prod_{i=1}^{r_W} \vert a_i\vert^{n}\prod_{j=1}^{r_V}\vert b_j\vert^{m-r_W}\prod_{i=1}^{r_W}\prod_{j=1}^{r_V}\Upsilon(a_ib_j^{-1})\nu_{\mathscr{S}}(\phi)\nu_{\mathscr{S}}(\phi').
		\end{equation}
	\end{enumerate}
\end{proposition}
\begin{proof}
	Note that in \cite[Lemma 3.1]{Xue22}, Xue proved a general result: Let $m$ be an integer and take $\phi,\phi'\in\mathscr{S}(\mathbb{R}^m)\widehat{\otimes}\mathscr{S}_{00}$, viewed as Schwartz functions valued in $\mathscr{S}_{00}$, and $\lambda=(\lambda_1, \cdots, \lambda_m)\in {\mathbb{R}}^m$. Then there is a seminorm $\nu$ on $\mathscr{S}({\mathbb{R}}^m)\widehat{\otimes}\mathscr{S}_{00}$ such that \begin{equation}\label{Xueint}
		\left\vert \int_{\mathbb{R}^m} \langle\phi(\lambda_1 x_1,\cdots,\lambda_m x_m),\phi'(x_1,\cdots,x_m)\rangle dx_1\cdots dx_m\right\vert\leq \prod_{i=1}^m\Upsilon(\lambda_i^{-1})\nu(\phi)\nu(\phi'). 
	\end{equation}   Together with the formulae (\ref{realaction}) and (\ref{quaternionaction}), we can deduce our result.
	
\end{proof}

\subsection{Weil representation and theta lifts}

Let $\omega_{\psi}$ be the Weil representation of dual pair $(G_1,G_2)$ of type I over $(K,\sharp)$. If $K=\mathbb{R}$ ($resp. \mathbb{H}$), let $\pi$ be an irreducible genuine representation of the double cover $\widehat{G}_1$ of $G_1$ (resp. an irreducible representation of $G_1$). Then the tensor product $\omega_{\psi}\otimes\pi$ is a $\widehat{G}_1\times\widehat{G}_2$-module, where $\widehat{G}_2$ is the double covering group corresponding to $G_2$ and acting by $\omega_{\psi}$ and $\widehat{G}_1$ acts by $\omega_{\psi}\otimes\pi$. The maximal isotropic quotient of $\omega_{\psi}$ with respect to $\pi$ has the form $\pi\boxtimes \Theta_{\psi}(\pi)$ for some smooth representation $\Theta_{\psi}(\pi)$ of $G_2$, which is either $0$ or of finite length. Let $\theta_{\psi}(\pi)$ be the maximal semi-simple quotient of $\Theta_{\psi}(\pi)$. It is known by Howe \cite{Howe89} that $\theta_{\psi}(\pi)$ is either zero or irreducible.

If $K=\mathbb{H}$, we regard $\mathrm{Sp}(W)$ as a subgroup of $\mathrm{Sp}_{2m}(\mathbb{C})$. We may denote 
\begin{equation}
	\dim W=2n
	\begin{cases}
		2n & \textit{ if } K=\mathbb{R}\\
		2m & \textit{ if } K=\mathbb{H}
	\end{cases}  
	\textit{ and }\dim V=
	\begin{cases}
		2n+1 & \textit{ if } K=\mathbb{R}\\
		2n & \textit{ if } K=\mathbb{H}
	\end{cases}   
\end{equation}

\begin{lemma}\label{cont} Let  $(G_1, G_2)$ be one of our type $I$ dual pair over $(K, \sharp)$ of equal rank $n$ with underlying spaces $(W, V)$. Let $\pi$ be an irreducible genuine tempered representation of $\widehat{\rm Sp}(W)$ (if $K=\bbR$) or an irreducible tempered representation of $\mathrm{Sp}(W)$ (if $K=\bbH$).
	For any $v, v'\in \pi$ and $\phi,\phi'\in \omega_{W, V, \psi}$,  there exists continuous semi-norms $\nu_{\pi}$ on $\pi$ and $\nu_{\mathscr{S}}$ on $\omega_{W, V, \psi}$ such that
	\[\left\vert\int_{\mathrm{Sp}(W)}(\omega_{W,V,\psi}(g,1)\phi,\phi')\overline{(\pi(g)v,v')}dg\right\vert\leq \nu_{\pi}(v)\nu_{\pi}(v')\nu_{\mathscr{S}}(\phi)\nu_{\mathscr{S}}(\phi').\]
\end{lemma}
\begin{proof}
	By estimations (\ref{matrixcoeff}) and (\ref{EST2}), there exists a continuous seminorm $\tilde{\nu}_{\pi}$ and positive constants $A,B$ such that for any $a\in A_W^+$ and $v,v'\in\pi$,
	$$\vert (\pi(a)v,v')\vert \leq A_1\delta^{\frac{1}{2}}_{P,\mathrm{Sp}(W)}(a)\sigma(a)^{B_1}\tilde{\nu}_{\pi}(v)\tilde{\nu}_{\pi}(v'). $$
	
	By (\ref{mcXue}) and (\ref{mcXue2}), there exists a continuous seminorm $\tilde{\nu}_{\mathscr{S}}$ such that for $a\in A_W^+$ and $\phi,\phi'\in\omega_{W,V,\psi}$, 
	$$\vert (\omega_{W,V,\psi}(a,1)\phi,\phi')\vert\leq \prod_{i=1}^{r_W} \vert a_i\vert^{\frac{\dim V}{2}}\tilde{\nu}_{\mathscr{S}}(\phi)\tilde{\nu}_{\mathscr{S}}(\phi'). $$
	
	Finally by the formula (\ref{deco}), the integral is bounded by
	\begin{equation}
		\begin{split}
			&\int_{ A_{W}^+}\delta_{P,\mathrm{Sp}(W)}(a)^{\frac{1}{2}}(1+\log\vert a_i\vert)^B\prod_{j=1}^{r_W}\vert a_j\vert^{\frac{\dim V}{2}}  da\\
			&\int_{K_1\times K_1}\tilde{\nu}_{\pi}(\pi(k_1)v)\tilde{\nu}_{\pi}(\pi(k_1^{'-1})v')\tilde{\nu}_{\mathscr{S}}(\omega_{W, V, \psi}(k_1,1)\phi)\tilde{\nu}_{\mathscr{S}}(\omega_{W, V, \psi}(k_1^{'-1},1)\phi')dk_1dk_1^{'},
		\end{split}
	\end{equation}
	where $B$ is a positive constant and $\tilde{\nu}_{\pi}$ (resp. $\tilde{\nu}_{\mathscr{S}}$) is a continuous semi-norm on $\pi$ (resp. $\omega_{W,V,\psi}$). 
	
	For any $a\in A_{W}^+$, $\delta_{P,\mathrm{Sp}(W)}(a)=\prod_{i=1}^{r_W}\vert a_i\vert^{2n+2-2i}$.
	The integral $$\int_{A_{W}^+}\prod\limits_{i=1}^{r_W}\vert a_i\vert^{-\frac{1}{2}(2n+2-2i)}(1-\sum\limits_{i=1}^{r_W}\log\vert a_i\vert)^B\prod_{j=1}^{r_W}\vert a_j\vert^{\frac{\dim V}{2}}da$$ converges. Since $K_1$ is compact, the integral
	\[\int_{K_1\times K_1}\tilde{\nu}_{\pi}(\pi(k_1)v)\tilde{\nu}_{\pi}(\pi(k_1^{'-1},1)v')\tilde{\nu}_{\mathscr{S}}(\omega_{W, V, \psi}(k_1,1)\phi)\tilde{\nu}_{\mathscr{S}}(\omega_{W, V, \psi}(k_1^{'-1},1)\phi')dk_1dk_1^{'}\] is bounded by 
	\[\mathrm{Vol}(K_1)^2\nu_{\pi}(v)\nu_{\pi}(v')\nu_{\mathscr{S}}(\phi)\nu_{\mathscr{S}}(\phi'),\] where $\nu_{\pi}(v)=\sup_{k_1\in K_1}\tilde{\nu}_{\pi}(\pi(k_1)v)$ and $\nu_{\mathscr{S}}(\phi)=\sup_{k_1\in K_1}\tilde{\nu}_{\mathscr{S}}(\omega_{W,V,\psi}(k_1,1)\phi)$. Each $\sup$ term defines a continuous semi-norm on the corresponding space by the uniform boundedness principle \cite[Theorem 33.1]{Tr67}. 
\end{proof}
\begin{proposition}Let $\pi$ be an irreducible genuine tempered representation of $\widehat{\rm Sp}(W)$ or an irreducible tempered representation of $\mathrm{Sp}(W)$. Take $v, v'\in \pi$ and $\phi, \phi'\in \omega_{W,V,\psi}$. The multilinear form\footnote{We ignore the identification of multilinear form and the linear form via the tensor product.}
	\begin{equation}\label{form}
		(v, v',\phi,\phi')\mapsto \int_{{\rm Sp}(W)}\overline{(\pi'(g)v, v')}\cdot (\omega_{W, V, \psi}(g,1)\phi, \phi') dg
	\end{equation}
	continuously extends to a linear form on $\bar{\pi}\widehat{\otimes}\pi\widehat{\otimes} \omega_{W, V, \psi}\widehat{\otimes}\bar{\omega}_{W, V, \psi}$. It is not identically zero if and only if $\theta_{W, V,\psi}(\pi)\neq 0$.
\end{proposition}
\begin{proof}The absolute convergence and continuity follow from Lemma \ref{cont}. The nonvanishing is proved by Gan, Qiu and Takeda in \cite{GQT14}.
\end{proof}
The integral (\ref{form}) defines a hermitian form on $\bar{\pi}\otimes \omega_{W,V,\psi}$. In fact, for any $\phi,\phi'\in\omega_{W,V,\psi}$ and $v,v'\in\pi$,
\begin{equation}
	\begin{split}
		\overline{\langle v\otimes\phi,v'\otimes\phi'\rangle}=
		&\overline{\int_{\rm Sp(W)}\overline{(\pi(g)v, v')}\cdot (\omega_{W, V, \psi}(g,1)\phi, \phi') dg}\\
		=&	\int_{\rm Sp(W)}(\pi(g)v, v')\cdot\overline{ (\omega_{W, V, \psi}(g,1)\phi, \phi')} dg\\
		=&\int_{\rm Sp(W)}\overline{(v',\pi(g) v)}\cdot (\phi',\omega_{W, V, \psi}(g,1) \phi) dg\\
		=&\int_{\rm Sp(W)}\overline{(\pi(g^{-1}) v',v)}\cdot (\omega_{W, V, \psi}(g^{-1},1) \phi',\phi) dg\\
		=&\int_{\rm Sp(W)}\overline{(\pi(g) v',v)}\cdot (\omega_{W, V, \psi}(g,1) \phi',\phi) dg\\
		=&\langle v'\otimes\phi',v\otimes\phi\rangle
	\end{split}
\end{equation}
which means (\ref{form}) defines a hermitian form on $\Theta_{W, V, \psi}(\pi)$. By \cite[Section 2]{He},  
this form is semi-positivity.
Moreover, we have the fact that: if $q$ is a nonzero semi-positive definite hermitian form on a vector space $X$, and $L$ is the radical of $q$, then $q$ descends to an inner product on $X/L$, still denote by $q$. To prove this, if there exists an $x\notin L$ such that $q(x,x)=0$, then take some $y\in X$, which satisfies $q(x,y)\neq 0$. For $t\in\mathbb{C}$, then we have\[q(tx+y,tx+y)=q(y,y)+2\mathrm{Re}(t)\cdot q(x,y).\] As $t$ is an arbitraty complex number and $q(x,y)\neq 0$, we conclude that for a well-chosen complex number $t$, $q(tx+y,tx+y)$ can be a negative real number, which is a contradiction to the semi-positivity of $q$.

Let $R$ be the radical of semi-positive hermitian form defined by (\ref{form}) as above. Then the nonzero semi-positive definite hermitian form $q$ defines an inner product on $\Theta_{W, V, \psi}(\pi)/R$. Therefore $\Theta_{W, V, \psi}(\pi)/R$ must be semisimple, and thus coincides with $\theta_{W, V, \psi}(\pi)$.
\begin{remark}
	It is important to study when the equality $\Theta_{W, V, \psi}(\pi)=\theta_{W, V, \psi}(\pi)$ holds.
\end{remark}
The explicit theta correspondence allows us to give the explicit matrix coefficients of $\theta_{W,V,\psi}(\pi)$ as follows.
\begin{proposition}Let  $(G_1, G_2)$ be one of our type $I$ dual pair over $(K, \sharp)$ of equal rank $n$ with underlying spaces $(W, V)$. Let $\pi$ be an irreducible genuine tempered representation of $\widehat{\rm Sp}(W)$ ( if $K=\bbR$) or an irreducible tempered representation of $\mathrm{Sp}(W)$ (if $K=\bbH$).
	Then the function $$\Phi_{\phi,\phi',v,v'}:h\in O(V)\mapsto\int_{\rm Sp(W)}\overline{(\pi(g)v, v')}\cdot (\omega_{W, V, \psi}(g,h)\phi, \phi') dg$$ defines a matrix coefficient of $\theta_{W, V, \psi}(\pi)$.
\end{proposition}

We also need the following proposition to simplify the computation.

\begin{proposition}\label{redense}Let $\Phi$ be the subspace of the matrix coefficients of  $\theta_{W,V,\psi}(\pi)$ generated by $\Phi_{\phi,\phi',v,v'}$, where $\phi$ and $\phi'$ range over a dense subspace of $\omega_{W,V,\psi}$, and $v$ and $v'$ range over a dense subspace of $\pi$. Then the space $\Phi$ is dense in the space of matrix coefficients of $\theta_{W,V,\psi}(\pi)$.
\end{proposition}
\begin{proof}
	Fix a surjective homomorphism $c:\omega_{W,V,\psi}\widehat{\otimes}\overline{\pi}\rightarrow\theta_{W,V,\psi}(\pi)$. The matrix coefficients of $\theta_{W,V,\psi}(\pi)$ are of the form $\langle \theta_{W,V,\psi}(\pi)(h)c(\phi,v),c(\phi',v')\rangle$ with $h\in O(V)$. Then the assertion follows from the surjectivity of $c$ and the density of $\omega_{W,V,\psi}\otimes\overline{\pi}$ in $\omega_{W,V,\psi}\widehat{\otimes}\overline{\pi}$.
\end{proof}

\section{Theta lifts for tempered representations}
In this paragraph, we use the estimations of the matrix coefficients of various representations established in the previous sections to prove our main theorem \ref{main}. To prove the theorem, it is suffit to show that the matrix coefficients of $\theta_{W,V,\psi}(\pi)$ are almost square integrable functions (i.e. $L^{2+\epsilon_0}(O(V))$ for all $\epsilon_0\in\mathbb{R}_{>0}$). By the proposition \ref{redense}, it suffices to prove that for any $\epsilon_0\in\mathbb{R}_{>0}$, for any $\phi,\phi'\in\omega_{W,V,\psi}$ and $v,v'\in \pi$, the integral 
\begin{equation*}
	\int_{O(V)}\left\vert \Phi_{\phi,\phi',v,v'}(h) \right\vert^{2+\epsilon_0}dh=\int_{O(V)}\left\vert\left(\int_{\mathrm{Sp}(W)}(\omega_{W,V,\psi}(g,h)\phi,\phi')\overline{(\pi(g)v,v')} dg\right)\right\vert^{2+\epsilon_0}dh
\end{equation*}
converges. In the following, we will prove a stronger condition:
the integral 
\begin{equation}\label{minusfinal}\int_{O(V)}\left(\int_{\mathrm{Sp}(W)}\left\vert(\omega_{W,V,\psi}(g,h)\phi,\phi')(\pi(g)v,v')\right\vert dg\right)^{2+\epsilon_0}dh\end{equation} converges.

Let $r_W$ and $r_V$ be the Witt index of $W$ and $V$ respectively. Note that, for our dual pair of type $I$ of equal rank $n$, if $K=\bbH$, the dimension of $W$ over $\bbH$ can be $n-1$ or $n$. In the following, if $K=\bbH$, we will assume $\dim W=\dim V=n$ and the other case works in the same way. 

\subsection{Reduction using the estimation of matrix coefficients}
Let ${\rm Sp}(W)=K_1A_W^+K_1$ and ${\rm O}(V)=K_2A_V^+K_2$ be the Cartan decomposition of ${\rm Sp}(W)$ and ${\rm O}(V)$ respectively. Let $g\in\mathrm{Sp}(W)$, $h\in \mathrm{O}(V)$, then there exists $a=(a_1, \cdots, a_{r_W})\in A_W^+$, $b=(b_1, \cdots, b_{r_V})\in A_V^+$, $k_1,k_1'\in K_1$ and $k_2,k_2'\in K_2$ such that $g=k_1ak_1'$ and $h=k_2bk_2'$.

For any $\phi,\phi'\in\omega_{W,V,\psi}$ and $v,v'\in \pi$, by the estimations (\ref{estmeta}) and the estimation of the matrix coefficient of Weil representation (see (\ref{estWeil}) and (\ref{estWeil2})), we deduce that there exists positive constants $A,B$ such that
\begin{equation}
	\begin{split}	
		&\left\vert(\omega_{W,V,\psi}(g,h)\phi,\phi')(\pi(g)v,v')\right\vert\\
		\leq & A\delta_{P, {\mathrm{Sp}(W)}}^{\frac{1}{2}}(a)\sigma(a)^B\prod_{i=1}^{r_W} \vert a_i\vert^{\frac{\dim V}{2}}\prod_{j=1}^{r_V}\vert b_j\vert^{\frac{\dim W}{2}-r_W}\prod_{i=1}^{r_W}\prod_{j=1}^{r_V}\Upsilon(a_ib_j^{-1}).
	\end{split}
\end{equation}
To simplify the notation, for $a\in A_W^+$ and $b\in A_V^+$, we set 
\[C_{W,V}(a,b)=\prod_{i=1}^{r_W} \vert a_i\vert^{\frac{\dim V}{2}}\prod_{j=1}^{r_V}\vert b_j\vert^{\frac{\dim W}{2}-r_W}\prod_{i=1}^{r_W}\prod_{j=1}^{r_V}\Upsilon(a_ib_j^{-1}).\]   
Together with the equation (\ref{deco}), we have
\begin{equation}
	\begin{split}
		&\int_{\mathrm{Sp}(W)}\left\vert(\omega_{W,V,\psi}(g,h)\phi,\phi')(\pi(g)v,v')\right\vert dg \\
		\leq& A\int_{\mathrm{Sp}(W)}\delta_{P,{\mathrm{Sp}(W)}}^{\frac{1}{2}}(a)\sigma(a)^B C_{W,V}(a,b)dg\\
		\leq& A\int_{A_{W}^+}\delta_{P,{\mathrm{Sp}(W)}}^{-1}(a)\int_{K_1\times K_1}\delta_{P,{\mathrm{Sp}(W)}}^{\frac{1}{2}}(a)\sigma(a)^B C_{W,V}(a,b)dk_1dadk'_1\\
		=&A\cdot{\rm Vol}(K_1)^2\cdot\int_{A_{W}^+}\delta_{P,{\mathrm{Sp}(W)}}^{-\frac{1}{2}}(a)\sigma(a)^B C_{W,V}(a,b)da.
	\end{split}
\end{equation}
Hence if we denote $A\cdot{\rm Vol}(K_1)^2$ by $A'$, then for any $\epsilon_0=2\epsilon>0$, using the equation (\ref{deco}) again, we have
\begin{equation}\label{Intbefore}
	\begin{split}
		&\int_{O(V)}\left(\int_{\mathrm{Sp}(W)}\left\vert(\omega_{W,V,\psi}(g,h)\phi,\phi')(\pi(g)v,v')\right\vert dg\right)^{2(1+\epsilon)}dh\\
		\leq& A' \int_{O(V)}\left(\int_{ A_{W}^+}\delta_{P,{\mathrm{Sp}(W)}}^{-\frac{1}{2}}(a)\sigma(a)^B C_{W,V}(a,b)da\right)^{2(1+\epsilon)}dh\\
		\leq & A'\int_{A_{V}^+}\delta_{P,{O(V)}}^{-1}(b)\int_{K_2\times K_2}\left(\int_{A_{W}^+}\delta_{P,{\mathrm{Sp}(W)}}^{-\frac{1}{2}}(a)\sigma(a)^B C_{W,V}(a,b)da\right)^{2(1+\epsilon)}dk_2dbdk'_2\\
		=& A'\cdot\mathrm{Vol}(K_2)^2\int_{A_{V}^+}\delta_{P,{O(V)}}^{-1}(b)\left(\int_{A_{W}^+}\delta_{P,{\mathrm{Sp}(W)}}^{-\frac{1}{2}}(a)\sigma(a)^B C_{W,V}(a,b)da\right)^{2(1+\epsilon)}db.
	\end{split}
\end{equation}

By the formula (\ref{height}), we have 
\[\sigma(a)\leq 1-\sum_{i=1}^{r_W}\log\vert a_i\vert\leq 1-\sum_{i=1}^{r_W}\log\vert a_i\vert-\sum_{j=1}^{r_V}\log\vert b_j\vert. \]	
Note that if $K=\bbH$, we realize the element of $\mathrm{GL}(1,\mathbb{H})$ as an element of $\mathrm{GL}(2,\mathbb{C})$.  
Thus the modular characters $\delta_{P,\mathrm{Sp}(W)}$ and $\delta_{P, {\rm O}(V)}$ are given by the following formula
\[\delta_{P,{\mathrm{Sp}(W)}}(a)= \prod_{i=1}^{r_W}\vert a_i\vert^{2n+2-2i};\] 
\[   \delta_{P,{O(V)}}(b)=\begin{cases}\prod_{j=1}^{r_V}\vert b_j\vert^{2n+1-2j}, & \text{ if } K=\bbR; \\ \prod_{j=1}^{r_V}\vert b_j\vert^{2n-2j}, & \text{ if } K=\bbH .\end{cases}\]

Thus, we have the integral	
\begin{equation}\label{intli}
	(\ref{Intbefore})=\begin{split}
		&\int_{ A_{W}^+\times A_{V}^+} \prod_{i=1}^{n} \vert a_i\vert^{(2i-1)(1+\epsilon)}\prod_{j=1}^{r_V}\vert b_j\vert ^{2j-2n-1}\prod_{k=1}^n\prod_{j=1}^{r_V}\Upsilon(a_kb_j^{-1})^{2+2\epsilon}\\ &(1-\sum_{i=1}^n\log\vert a_i\vert-\sum_{j=1}^{r_V}\log\vert b_j\vert)^{B(2+2\epsilon)}dadb.
	\end{split}
\end{equation} 
if $K=\mathbb{R}$ and 
\begin{equation}\label{intli2}
	(\ref{Intbefore})=\begin{split}
		&\int_{ A_{W}^+\times A_{V}^+} \prod_{i=1}^{r_W} \vert a_i\vert^{2(i-1)(1+\epsilon)}\prod_{j=1}^{r_V}\vert b_j\vert ^{(2j-2n)+2(1+\epsilon)(n-r_W)}\prod_{k=1}^{r_W}\prod_{j=1}^{r_V}\Upsilon(a_kb_j^{-1})^{2+2\epsilon}\\ &(1-\sum_{i=1}^{r_W}\log\vert a_i\vert-\sum_{j=1}^{r_V}\log\vert b_j\vert)^{B(2+2\epsilon)}dadb.
	\end{split}
\end{equation} if $K=\mathbb{H}$.

To prove the convergence of the integral (\ref{com}), it suffices to show the integrals (\ref{intli}) and (\ref{intli2}) are convergent.

\subsection{Proof of the convergence of the integral (\ref{intli})}
We prove the convergence of the integral (\ref{intli}), this method also works for the convergence of the integral (\ref{intli2}).  Let $(p_1,\cdots,p_{r_V+1})$ be a $(r_V+1)$-tuple of non-negative integers such that 
\[p_1+\cdots p_{r_V+1}=n.\] Let $S_{p_1,\cdots,p_{r_V+1}}$ be the subset of  $A_{W}^+\times A_{V}^+$, defined by the condition
\begin{equation}\label{order}
	\begin{split}
		&\vert a_1\vert\leq\cdots\leq \vert a_{p_1}\vert\leq\vert b_1\vert \\ \leq  &\vert a_{p_1+1}\vert\leq \cdots\leq\vert a_{p_1+p_2}\vert \leq\vert b_2\vert 
		\leq\vert a_{p_1+p_2+1}\vert\leq\cdots\leq\vert a_{p_1+\cdots+p_{r_V}}\vert \\\leq &\vert b_{r_V}\vert\leq \vert a_{p_1+\cdots+p_{r_V}+1}\vert \leq\cdots\leq\vert a_{p_1+\cdots+p_{r_V+1}}\vert
		\leq 1.
	\end{split}
\end{equation}
We can break the domain $A_{W}^+\times A_{V}^+$ of the integral (\ref{intli}) by $S_{p_1,\cdots,p_{r_V+1}}$, and it suffices to show that over each region $S_{p_1,\cdots,p_{r_V+1}}$, the integral (\ref{intli}) converges.
We will use the following simple lemma to conclude its convergence.	
\begin{lemma}\label{com}
	Let $N$ be a natural number. Let $s_1,\cdots,s_N$ and $B$ be real numbers. If $s_1+\cdots+s_i>0$ for all $1\leq i\leq N$, then the integral $$\int_{\vert x_1\vert\leq\cdots\leq\vert x_N\vert\leq 1}\vert x_1\vert^{s_1}\cdots\vert x_N\vert^{s_N}(1-\sum_{i=1}^N\log\vert x_i\vert)^{B}dx_1\cdots dx_N$$ converges.
\end{lemma}
Note that in a fixed region $S_{p_1,\cdots,p_{r_V+1}}$, we have \begin{equation*}
	\begin{split}
		\prod_{i=1}^n\prod_{j=1}^{r_V}\Upsilon(a_ib_j^{-1})
		=\prod_{j=1}^{r_V}\left(\left\vert \prod_{i=1}^{p_{j+1}}a_{i+\sum_{k=1}^jp_k}\right\vert^{-j}\cdot\left\vert b_j\right\vert^{n-(\sum_{k=1}^{j}p_k)}\right).
	\end{split}
\end{equation*}
We rearrange the terms in the integral (\ref{intli}) with respect to the order given by the condition (\ref{order}). To prove the integral (\ref{intli}) converges, it suffices to prove the integral (\ref{intli}) on region $S_{p_1,\cdots,p_{r_V+1}}$ satisfies the condition of lemma \ref{com} with respect to this order.

For $0\leq t\leq p_{j+1},1\leq j\leq r_V$, we check the sum of the exponents in the integral (\ref{intli}) up to $a_{p_1+\cdots+p_j+t}$ :
\begin{enumerate}
	\item The sum of the exponents of $a_i(1\leq i\leq p_1+\cdots+p_j+t)$:
	$$(1+3+\cdots+(2(p_1+\cdots+p_j+t)-1))(1+\epsilon)-(p_2+2p_3+\cdots+(j-1)p_j+jt)(2+2\epsilon)$$
	\item The sum of the exponents of $b_i(1\leq i\leq j)$:
	\begin{equation*}
		\begin{split}
			2(1+\cdots+j)-j(2n+1)+((n-p_1)+\cdots+(n-p_1-\cdots-p_j))(2+2\epsilon).
		\end{split}
	\end{equation*}
\end{enumerate} 
Summing these two terms, we get
$$(p_1+\cdots+p_j+t)^2+\epsilon((p_1+\cdots+p_j+t)^2+2j(n-p_1-\cdots-p_j-t))>0.$$ The same type of verification shows that the sum of the exponents up to $b_j$ is positive. Hence the integral (\ref{intli}) satisfies the condition of the lemma \ref{com}. As a consequence, the integral (\ref{minusfinal}) converges.

\section*{Acknowledgement}

This note is based on a disussion with Xue Hang. The authors would like to express their gratitude to Xue Hang for explaining his works on unitary groups to us. The second author would like to thank Wenwei Li and Fang Gao for communications on representations of metaplectic groups.

\end{document}